\documentclass{amsart}
\usepackage{amsmath}
\usepackage{amssymb}
\usepackage{amsthm}
\usepackage{graphicx}
\usepackage{color}
\usepackage{bm}

\newtheorem{theorem}{Theorem}[section]
\newtheorem{proposition}[theorem]{Proposition}

\theoremstyle{definition}
\newtheorem{definition}[theorem]{Definition}
\newtheorem{example}[theorem]{Example}

\newtheorem{remark}[theorem]{Remark}

\begin{document}

\title[]{Dehn colorings and vertex-weight invariants for spatial graphs}

\author[K.~Oshiro]{Kanako Oshiro}
\address{Department of Information and Communication Sciences, Sophia University, Tokyo 102-8554, Japan}
\email{oshirok@sophia.ac.jp}

\author[N.~Oyamaguchi]{Natsumi Oyamaguchi}
\address{Department of Teacher Education, Shumei University, Chiba 276-0003, Japan}
\email{p-oyamaguchi@mailg.shumei-u.ac.jp}



\date{\today}

\maketitle

\begin{abstract}
In this paper, we study Dehn colorings for spatial graphs, and give a family of spatial graph invariants that are called {\it vertex-weight invariants}.
We give some examples of spatial graphs that can be distinguished by a vertex-weight invariant, whereas distinguished by neither their constituent links nor the number of Dehn colorings.
\end{abstract}

\section*{Introduction}
Fox colorings for classical links have been used by various studies in knot theory, see \cite{HararyKauffman, Przytycki95, Satoh07} for example. 
In \cite{IshiiYasuhara97}, Fox colorings for spatial graph diagrams were studied with two kinds of vertex conditions, 
and in \cite{Oshiro12}, vertex conditions for Fox colorings of spatial graph diagrams were completely classified with ``some invariants for an equivalence relation on $\sum_{n\in 2\mathbb Z_{+}}\mathbb Z_p^n$".
We note that the classification gives the maximum generalization for Fox colorings of spatial graph diagrams unless we change the fundamental definition that a {\it Fox $p$-coloring} of a diagram $D$  is a map $C:\{\mbox{arcs of $D$}\}\to \mathbb Z_p$ satisfying the crossing condition depicted in Figure~\ref{arccoloring}.  
Dehn colorings, namely region colorings by $\mathbb Z_p$, for classical links have been also studied in knot theory, see \cite{CarterSilverWilliams13, MadausNewmanRussell17, Niebrzydowski0} for example. In particular, in \cite{CarterSilverWilliams13}, some relation between Fox colorings and Dehn colorings was given.

In this paper, we study Dehn colorings for spatial graph diagrams, and we discuss about ``some invariants for an equivalence relation on $\sum_{n\in 2\mathbb Z_{+}}\mathbb Z_p^n$" related to Dehn colorings of spatial graph diagrams.
Note that as in the case of Fox colorings of spatial graph diagrams, the invariants can be used for the classification of vertex conditions for Dehn colorings of spatial graph diagrams, which will be studied in our next paper \cite{OshiroOyamaguchi20(2)}.
Furthermore, we show each invariant for the equivalence relation on $\sum_{n\in 2\mathbb Z_{+}}\mathbb Z_p^n$ gives a spatial graph invariant called a {\it vertex-weight invariant}. We give some examples of spatial graphs that can be distinguished by a vertex-weight invariant, whereas distinguished by neither their constituent links nor the number of Dehn colorings. Note that the notion of a vertex-weight invariant discussed in this paper can be also applied for the case of Fox colorings of spatial graphs.
\begin{figure}[h]
  \begin{center}
    \includegraphics[clip,width=3cm]{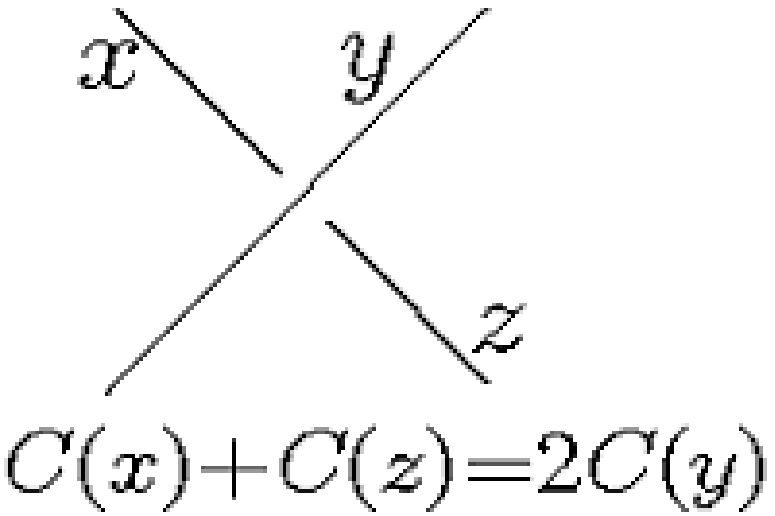}
    \caption{}
    \label{arccoloring}
  \end{center}
\end{figure}

This paper is organized as follows:
In Section~1, we introduce an equivalence relation on $\sum_{n\in 2\mathbb Z_{+}}\mathbb Z_p^n$, and discuss some invariants under the equivalence relation.
In Section~2, we review the definitions of spatial graphs and their diagrams. 
In Section~3, a Dehn coloring of a spatial Euler graph diagram is defined. 
Section~4 is devoted to the study of vertex-weight invariants of spatial Euler graphs, and presents some example of spatial Euler graphs that can be distinguished by a vertex-weight invariant, whereas distinguished by neither their constituent links nor the number of Dehn colorings.
Section~5 deals with the case of spatial graphs each of which includes an odd-valent vertex, and presents some example of spatial graphs that can be distinguished by a vertex-weight invariant, whereas distinguished by neither their constituent links nor the number of Dehn colorings.

\section{Invariants of an equivalence relation on $\sum_{n\in 2\mathbb Z_{+}}\mathbb Z_p^n$}
Throughout this paper, $\mathbb Z_+$ means the set of positive integers, $\mathbb Z_{\geq q}$ means the set of integers greater than or equal to $q$, and $\mathbb Z_p=\{0,1, \ldots, p-1\}$ means the cyclic group $\mathbb Z/p\mathbb Z$.

From now on, let $p\in \mathbb Z_{\geq 2}$, and  put $U_p=\bigcup_{n\in 2\mathbb Z_+} \mathbb Z^n_p$.

\begin{definition}\label{def:Requiv}
Two elements $\displaystyle \boldsymbol{a}, \boldsymbol{b} \in U_p$ are {\it equivalent} ($\boldsymbol{a} \sim \boldsymbol{b}$) if $\boldsymbol{a}$ and $ \boldsymbol{b}$  are related by a finite sequence of the following transformations:
\begin{itemize}
\item[(Op1)] $(a_1,  \ldots , a_n) \longrightarrow (a_2, \ldots , a_n, a_1 )$,
\item[(Op2)] $(a_1,  \ldots , a_n) \longrightarrow (a, a_2+(-1)^2(a_1-a), \ldots , a_i +(-1)^i (a_1-a), \ldots ,  a_n +(-1)^n(a_1-a) )$ for $a\in \mathbb Z_p$,
\item[(Op3)] $(a_1,  \ldots , a_n) \longrightarrow (a, a_1-a_2 + a, \ldots ,  a_1-a_i + a, \ldots ,  a_1 -a_n+a )$ for $a\in \mathbb Z_p$,
\item[(Op4)]  $(a_1,  \ldots , a_n) \longrightarrow (a_1, -a_1 + a_2 + a_3, a_3 , \ldots  , a_n )$ when $n>3$.
\end{itemize}
\end{definition}

\begin{remark}
The inverse of (Op1) is (Op1)${}^{n-1}$. The inverse of (Op2)  is (Op2) for $a=a_1 \in \mathbb Z_p$. The inverse of (Op3)  is (Op3) for $a=a_1 \in \mathbb Z_p$. The inverse of (Op4)  is (Op4)${}^{p-1}$. 
\end{remark}

Put  $\bm{a}=(a_1, \ldots ,a_n)$. We define $\displaystyle \tau_p: U_p \longrightarrow \mathbb{Z}$ by
\[
\tau_p \big( \bm{a} \big)=\max\left\{ k \in \{1,\ldots,p\} ~ {\Bigg |} ~
\begin{array}{l}
 \ k|p, \\
 a_1+a_2 \equiv a_2+a_3 \equiv \cdots \equiv a_n+a_1 \pmod{k}
 \end{array}
  \right\}.
  \]

Suppose $p$ is an even integer. 
We define
$\displaystyle \varepsilon_p: U_p \longrightarrow \mathbb{Z} \cup \{\infty\}$ by
\[
\varepsilon_p \big( \bm{a}\big)=
\begin{cases}
0 & \mbox{if } a_1+a_2 \equiv \cdots \equiv a_n+a_1 \equiv 0  \pmod{2},\\
1 & \mbox{if } a_1+a_2 \equiv  \cdots \equiv a_n+a_1 \equiv 1  \pmod{2},\\
\infty & {\rm otherwise.}
\end{cases}
\]
We define
$\displaystyle \mu_p: U_p \longrightarrow \mathbb{Z}$ by
\[
\mu_p \big( \bm{a} \big)=E( a_1+a_2, \ldots, a_n+a_1) -O  (a_1+a_2, \ldots ,a_n+a_1),
\]
where 
\[E\big( \bm{a}\big)=\# \{ i \in \{1,\ldots, n\} \mid a_i \equiv 0 \pmod{2}\}\]
 and
\[O\big( \bm{a} \big)=\# \{ i \in \{1,\ldots, n\} \mid a_i \equiv 1 \pmod{2}\}.\]
For $\tau \in \{1, \ldots ,p\}$ such that $\tau \equiv 0 \pmod{2}$, $\tau|p$ and $\displaystyle \frac{p}{\tau} \equiv 0 \pmod{2}$, define $\displaystyle \mu_{p, \tau}: U_p \longrightarrow \mathbb{Z} \cup\{\infty\}$ by
\[
\mu_{p,\tau}\big( \bm{a} \big)
=
\begin{cases}\displaystyle
~ {\Bigg |} ~
\mu_{\frac{p}{\tau}}  \Big( \frac{a_{1}-a_1}{\tau}, \ldots , \frac{a_{2j-1}-a_1}{\tau},\frac{a_{2j}-a_2}{\tau}, \ldots , \frac{a_n-a_2}{\tau} \Big) 
~ {\Bigg |} & \mbox{if } \tau_p(\bm{a})=\tau,\\
\hspace{30mm} \infty & {\rm otherwise.}
\end{cases}
\]

\begin{remark}
$\mu_{p, \tau}$ is well-defined since $a_{2j-1}-a_1 \equiv 0 \pmod{\tau}$ and $a_{2j}- a_2 \equiv 0 \pmod{\tau}$ for $j \in \{2, \ldots ,\frac{n}{2}\}$ when $\tau_p\big( \bm{a} \big)=\tau$.
\end{remark}

\noindent We have the following theorems.

\begin{theorem}
Let ${\bm a}$ and ${\bm b}$ be two equivalent elements of $U_p$. Then it holds that $\tau_p(\bm{a})=\tau_p(\bm{b})$.
\end{theorem}

\begin{proof}
It suffices to show that $\tau_p$ is an invariant under the transformations  (Op1)-(Op4)  in Definition~\ref{def:Requiv}.

(Op1) It is easy to see that
\[
\tau_p(a_2, \ldots , a_n, a_1)=\tau_p(a_1, \ldots , a_n)
\]
 for $(a_1, \ldots , a_n) \in U_p$.

(Op2) For $(a_1, \ldots , a_n) \in U_p$, $a \in \mathbb{Z}_p$ and $k \in \{1, 2,\ldots, p\}$ such that $k|p$, we have
\begin{align*}
&a_1+a_2 \equiv a_2+a_3 \equiv \cdots \equiv a_n+a_1 \pmod{k}\\
&\iff a+a_2+(-1)^2(a_1-a) \equiv  \cdots \equiv a_i+(-1)^i(a_1-a)+a_{i+1}+(-1)^{i+1}(a_1-a)\\
& \hspace{10mm} \equiv \cdots \equiv  a_n+(-1)^n(a_1-a)+a \pmod{k}.
\end{align*}
Hence 
\[\tau_p\big(a_1, \ldots , a_n\big)=\tau_p\big(a, a_2+(a_1-a), \ldots , a_i+(-1)^i(a_1-a), \ldots , a_{n}+(a_1-a)\big)
\]
 holds.

(Op3) For $(a_1, \ldots , a_n) \in U_p$, $a \in \mathbb{Z}_p$ and $k \in \{1, 2,\ldots, p\}$ such that $k|p$, we have
\begin{align*}
&a_1+a_2 \equiv a_2+a_3 \equiv \cdots \equiv a_n+a_1 \pmod{k}\\
&\iff 2a_1-(a_1+a_2)+2a \equiv  \cdots \equiv 2a_1-(a_i+a_{i+1})+2a\\
& \hspace{10mm} \equiv \cdots \equiv  2a_1-(a_n+a_1)+2a \pmod{k}.\\
&\iff a+(a_1-a_2+a) \equiv  \cdots \equiv (a_1-a_i+a)+(a_1-a_{i+1}+a)\\
& \hspace{10mm} \equiv \cdots \equiv  (a_1-a_n+a)+a \pmod{k}.
\end{align*}
Hence 
\[\tau_p(a_1, \ldots , a_n)=\tau_p(a, a_1-a_2+a, \ldots , a_1-a_i+a, \ldots ,a_1-a_n+a)
\]
 holds.
 
 (Op4) For $(a_1, \ldots , a_n) \in \displaystyle U_p$ and $k \in \{1, 2,\ldots, p\}$ such that $n>3$ and $k|p$, since 
 \[a_1+a_2 \equiv a_2+a_3 \pmod{k} \iff -(a_1+a_2)+(a_2+a_3) \equiv 0 \pmod{k},
 \]
we have
\begin{align*}
& a_1+a_2 \equiv a_2+a_3 \equiv \cdots \equiv a_n+a_1 \pmod{k}\\
&\iff a_1+a_2+\{-(a_1+a_2)+(a_2+a_3)\} \equiv a_2+a_3+\{-(a_1+a_2)+(a_2+a_3)\}\\
&\hspace{9mm} \equiv a_3 + a_4 \equiv \cdots \equiv a_n+a_1 \pmod{k}\\ 
&\iff a_2+a_3 \equiv (-a_1+a_2+a_3)+a_3 \equiv a_3+a_4 \equiv \cdots \equiv a_n+a_1 \pmod{k}\\ 
&\iff a_1+(-a_1+a_2+a_3) \equiv (-a_1+a_2+a_3)+a_3 \equiv \cdots \equiv a_n+a_1 \pmod{k}.
\end{align*}
Hence 
\[\tau_p(a_1, \ldots , a_n)=\tau_p(a_1, -a_1 + a_2 + a_3, a_3 , \ldots , a_n )
\]
 holds.
 \end{proof}

\begin{theorem}Suppose $p$ is an even integer.
Let ${\bm a}$ and ${\bm b}$ be two equivalent elements of $U_p$. Then it holds that $\varepsilon_p(\bm{a})=\varepsilon_p(\bm{b})$.
\end{theorem}

\begin{proof}
It suffices to show that $\varepsilon_p$ is an invariant under the transformations (Op1)-(Op4) in Definition~\ref{def:Requiv}.

(Op1) It is easy to see that 
\[
\varepsilon_p(a_2, \ldots , a_n, a_1)=\varepsilon_p(a_1, \ldots , a_n)
\]
 for $(a_1, \ldots , a_n) \in \displaystyle U_p$.

(Op2) For $(a_1, \ldots , a_n) \in  U_p$ and $a \in \mathbb{Z}_p$, we have
\begin{align*}
&\varepsilon_p(a_1, \ldots , a_n)=\varepsilon \\
&\iff {}^\forall i \in \{1, \ldots , n\}, a_i+a_{i+1} \equiv \varepsilon \pmod{2}\\
&\iff {}^\forall i \in \{1, \ldots , n\}, \big(a_i+(-1)^i (a_1-a)\big)+\big(a_{i+1}+(-1)^{i+1} (a_1-a)\big) \equiv \varepsilon \pmod{2}\\
&\iff \varepsilon_p\big(a, a_2+(-1)^2(a_1-a), \ldots , a_i+(-1)^i(a_1-a), \ldots , a_{n}+(-1)^n(a_1-a)\big)=\varepsilon
\end{align*}
for $\varepsilon \in\{0,1\}$, where $a_{n+1}=a_1$.

(Op3) For $(a_1, \ldots , a_n) \in U_p$ and $a \in \mathbb{Z}_p$, we have
\begin{align*}
&\varepsilon_p(a_1, \ldots , a_n)=\varepsilon\\
&\iff {}^\forall i \in \{1, \ldots , n\}, a_i+a_{i+1} \equiv \varepsilon \pmod{2}\\
&\iff {}^\forall i \in \{1, \ldots , n\}, \big(a_1-a_i+a\big)+\big(a_1-a_{i+1}+a\big) \equiv a_i+a_{i+1} \equiv \varepsilon \pmod{2}\\
&\iff \varepsilon_p (a, a_1-a_2+a, \ldots , a_1-a_i+a, \ldots ,a_1-a_n+a)=\varepsilon
\end{align*}
for $\varepsilon \in\{0,1\}$, where $a_{n+1}=a_1$.

(Op4) For $(a_1, \ldots , a_n) \in \displaystyle U_p$ such that $n>3$, we have
\begin{align*}
&\varepsilon_p(a_1, \ldots , a_n)=\varepsilon\\
&\iff {}^\forall i \in \{1, \ldots , n\}, a_i+a_{i+1} \equiv \varepsilon \pmod{2}\\
&\iff a_1+(-a_1+a_2+a_3)=a_2+a_3 \equiv \varepsilon \pmod{2},\\
& \hspace{10mm} (-a_1+a_2+a_3)+a_3 \equiv a_1+a_2 \equiv \varepsilon \pmod{2},\\
& \hspace{10mm}  {}^\forall i \in \{3, \ldots , n\}, a_i+a_{i+1} \equiv \varepsilon \pmod{2}\\
&\iff \varepsilon_p (a_1, -a_1+a_2+a_3, a_3, \ldots, a_n)=\varepsilon
\end{align*}
for $\varepsilon \in\{0,1\}$, where $a_{n+1}=a_1$.

\end{proof}

\begin{theorem}
Suppose $p$ is an even integer.
Let ${\bm a}$ and ${\bm b}$ be two equivalent elements of $U_p$. Then it holds that $\mu_p(\bm{a})=\mu_p(\bm{b})$.
\end{theorem}

\begin{proof}
It suffices to show that $\mu_p$ is an invariant under the transformations  (Op1)-(Op4)  in Definition~\ref{def:Requiv}.

(Op1) It is easy to see that 
\[\mu_p (a_2, \ldots , a_n, a_1) =\mu_p(a_1, \ldots , a_n) 
\]
 for $(a_1, \ldots , a_n) \in \displaystyle U_p $.

(Op2) For $(a_1, \ldots , a_n) \in \displaystyle U_p$ and $a \in \mathbb{Z}_p$, we have
\begin{align*}
&\mu_p(a_1, \ldots , a_n)\\
&=E(a_1+a_2, \ldots ,a_n+a_1)-O(a_1+a_2, \ldots ,a_n+a_1)\\
&=E\big(a+a_2+(a_1-a), \ldots ,(a_i+(-1)^i(a_1-a))+(a_{i+1}+(-1)^{i+1}(a_1-a)), \\
&\hspace{9cm}\ldots, a_n+(a_1-a)+a\big)\\
& \quad -O\big(a+a_2+(a_1-a), \ldots ,(a_i+(-1)^i(a_1-a))+(a_{i+1}+(-1)^{i+1}(a_1-a)), \\
&\hspace{9cm} \ldots, a_n+(a_1-a)+a\big)\\
& =\mu_p(a, a_2+(a_1-a), \ldots , a_i +(-1)^i (a_1-a), \ldots ,  a_n +(a_1-a) ).
\end{align*}

(Op3) For $(a_1, \ldots , a_n) \in \displaystyle U_p$ and $a \in \mathbb{Z}_p$, we have
\begin{align*}
&\mu_p(a_1, \ldots , a_n)\\
&=E(a_1+a_2, \ldots ,a_n+a_1)-O(a_1+a_2, \ldots ,a_n+a_1)\\
&=E\big(2a_1-(a_1+a_2)+2a, \ldots ,2a_1-(a_i+a_{i+1})+2a, \ldots, 2a_1-(a_n+a_1)+2a\big)\\
& \quad -O\big(2a_1-(a_1+a_2)+2a, \ldots ,2a_1-(a_i+a_{i+1})+2a, \ldots, 2a_1-(a_n+a_1)+2a\big)\\
&=E\big(a+(a_1-a_2+a), \ldots ,(a_1-a_i+a)+(a_1-a_{i+1}+a), \ldots, (a_1-a_n+a)+a\big)\\
& \quad -O\big(a+(a_1-a_2+a), \ldots ,(a_1-a_i+a)+(a_1-a_{i+1}+a), \ldots, (a_1-a_n+a)+a\big)\\
& =\mu_p(a, a_1-a_2 + a, \ldots ,  a_1-a_i + a, \ldots ,  a_1 -a_n+a ).
\end{align*}

(Op4) For $(a_1, \ldots , a_n) \in  U_p $ such that $n>3$, we have
\begin{align*}
&\mu_p(a_1, \ldots , a_n)\\
&=E(a_1+a_2, \ldots ,a_n+a_1)-O(a_1+a_2, \ldots ,a_n+a_1)\\
&=E\big(a_2+a_3, -(a_1+a_2)+2(a_2+a_3), a_3+a_4,\ldots ,a_n+a_1\big)\\
& \quad -O\big(a_2+a_3, -(a_1+a_2)+2(a_2+a_3), a_3+a_4,\ldots ,a_n+a_1\big)\\
&=E\big(a_1+(-a_1+a_2+a_3), (-a_1+a_2+a_3)+a_3, a_3+a_4,\ldots ,a_n+a_1\big)\\
&\quad -O\big(a_1+(-a_1+a_2+a_3), (-a_1+a_2+a_3)+a_3, a_3+a_4,\ldots ,a_n+a_1\big)\\
&=\mu_p (a_1, -a_1 + a_2 + a_3, a_3 , \ldots , a_n ).
\end{align*}

\end{proof}

\begin{theorem}
Suppose $p$ is an even integer.
Let $\tau \in \{1, \ldots ,p\}$ such that $\tau \equiv 0 \pmod{2}$, $\tau|p$ and $\displaystyle \frac{p}{\tau} \equiv 0 \pmod{2}$. Let ${\bm a}$ and ${\bm b}$ be two equivalent elements of $\displaystyle U_p$. Then it holds that $\mu_{p, \tau}(\bm{a})=\mu_{p, \tau}(\bm{b})$.
\end{theorem}

\begin{proof}
It suffices to show that $\mu_{p, \tau}$ is an invariant under the transformations  (Op1)-(Op4)   in Definition~\ref{def:Requiv}.

(Op1) It is easy to see that 
\[
\mu_{p, \tau}(a_2, \ldots , a_n, a_1)=\mu_{p, \tau} (a_1, \ldots , a_n)
\]
 for $(a_1, \ldots , a_n) \in \displaystyle U_p$.

(Op2) For $(a_1, \ldots , a_n) \in \displaystyle U_p$ and $a \in \mathbb{Z}_p$, we have
\begin{align*}
&\mu_{p, \tau}(a_1, \ldots , a_n)\\
&=~ {\Big |} ~
\mu_{\frac{p}{\tau}} \Big( \frac{a_1-a_1}{\tau} , \ldots , \frac{a_{2j-1}-a_1}{\tau},\frac{a_{2j}-a_2}{\tau}, \ldots , \frac{a_n-a_2}{\tau} \Big) 
~ {\Big |}\\
&=~ {\Big |} ~
\mu_{\frac{p}{\tau}} \Big( \ldots , \frac{(a_{2j-1}-(a_1-a))-a}{\tau},\frac{(a_{2j}+(a_1-a))-(a_2+(a_1-a))}{\tau}, \ldots  \Big)
~ {\Big |}\\
&=\mu_{p, \tau}\big(a, a_2+(a_1-a), \ldots , a_i +(-1)^i (a_1-a), \ldots ,  a_n +(a_1-a) \big).
\end{align*}

(Op3) For $(a_1, \ldots , a_n) \in \displaystyle U_p$ and $a \in \mathbb{Z}_p$, we have
\begin{align*}
&\mu_{p, \tau}(a_1, \ldots , a_n)\\
&=~ {\Big |} ~
\mu_{\frac{p}{\tau}}\Big( \frac{a_1-a_1}{\tau}, \ldots , \frac{a_{2j-1}-a_1}{\tau},\frac{a_{2j}-a_2}{\tau}, \ldots , \frac{a_n-a_2}{\tau}  \Big)
~ {\Big |}\\
&=~ {\Big |} ~
\mu_{\frac{p}{\tau}} \Big(  \frac{a_1-a_1}{\tau},  \ldots , \frac{a_1-a_{2j-1}}{\tau},\frac{a_2-a_{2j}}{\tau}, \ldots ,\frac{a_2-a_n}{\tau}    \Big)
~ {\Big |}\\
&=~ {\Big |} ~
\mu_{\frac{p}{\tau}} \Big( \ldots , \frac{(a_1-a_{2j-1}+a)-a}{\tau},\frac{(a_1-a_{2j}+a)-(a_1-a_2+a)}{\tau}, \ldots  \Big)
~ {\Big |}\\
&=\mu_{p, \tau}(a, a_1-a_2 + a, \ldots ,  a_1-a_i + a, \ldots ,  a_1 -a_n+a ).
\end{align*}

(Op4) For $(a_1, \ldots , a_n) \in U_p$ such that $n>3$, we have
\begin{align*}
&\mu_{p, \tau}(a_1, \ldots , a_n)\\
&=~ {\Big |} ~
\mu_{\frac{p}{\tau}}  \Big( \frac{a_1-a_1}{\tau}, \ldots , \frac{a_{2j-1}-a_1}{\tau},\frac{a_{2j}-a_2}{\tau}, \ldots , \frac{a_n-a_2}{\tau} \Big)
~ {\Big |}\\
&=~ {\Big |} ~E\Big(\ldots , \frac{a_{2j-1}-a_1}{\tau}+\frac{a_{2j}-a_2}{\tau}, \frac{a_{2j}-a_2}{\tau}+\frac{a_{2j+1}-a_1}{\tau}, \ldots\Big)\\
&\quad -O\Big(\ldots , \frac{a_{2j-1}-a_1}{\tau}+\frac{a_{2j}-a_2}{\tau}, \frac{a_{2j}-a_2}{\tau}+\frac{a_{2j+1}-a_1}{\tau}, \ldots\Big)~ {\Big |} ~\\
&=~ {\Big |} ~E\Big(\ldots , \frac{a_{2j-1}-a_1}{\tau}+(\frac{a_{2j}-a_2}{\tau}-\frac{a_{3}-a_1}{\tau}), \\
&\hspace{6cm}(\frac{a_{2j}-a_2}{\tau}-\frac{a_{3}-a_1}{\tau})+\frac{a_{2j+1}-a_1}{\tau}, \ldots\Big)\\
&\quad -O\Big(\ldots , \frac{a_{2j-1}-a_1}{\tau}+(\frac{a_{2j}-a_2}{\tau}-\frac{a_{3}-a_1}{\tau}), \\
&\hspace{6cm}(\frac{a_{2j}-a_2}{\tau}-\frac{a_{3}-a_1}{\tau})+\frac{a_{2j+1}-a_1}{\tau}, \ldots\Big)~ {\Big |} ~\\
&=~ {\Big |} ~
\mu_{\frac{p}{\tau}} \Big(\frac{a_1-a_1}{\tau},  \ldots , \frac{a_{2j-1}-a_1}{\tau}, \frac{a_{2j}-a_2}{\tau}-\frac{a_{3}-a_1}{\tau}, \ldots  , \frac{a_{n}-a_2}{\tau}-\frac{a_{3}-a_1}{\tau}\Big)
~ {\Big |}\\
&=~ {\Big |} ~
\mu_{\frac{p}{\tau}} \Big( \frac{a_1-a_1}{\tau} , \ldots , \frac{a_{2j-1}-a_1}{\tau}, \frac{a_{2j}-(-a_1+a_2+a_3)}{\tau}, \\
&\hspace{8cm} \ldots , \frac{a_{n}-(-a_1+a_2+a_3)}{\tau}\Big)
~ {\Big |}\\
&=\mu_{p, \tau}(a_1, -a_1 + a_2 + a_3, a_3 , \ldots , a_n ).
\end{align*}

\end{proof}

\section{Spatial graph diagrams}
A {\it spatial graph} is a graph  embedded in $\mathbb{R}^{3}$. 
We call a spatial graph each of whose vertices is  of even valence a {\it spatial Euler graph}. 
In this paper, a spatial graph means an unoriented spatial graph.
Two spatial graphs are {\it equivalent} if we can deform by an ambient isotopy of $\mathbb R^3$ one onto the other. 
A {\it diagram} of a spatial graph $G$ is an image of $G$ by a regular projection onto $\mathbb R^2$ with a height information at each crossing point. 
It is known that two spatial graph diagrams represent an equivalent spatial graph if and only if they are related by a finite sequence of the Reidemeister moves of type I-V depicted in Figure~\ref{Rmoves}. 
We call each connected component of complementary regions of a diagram a {\it region} of the diagram.
\begin{figure}[h]
  \begin{center}
    \includegraphics[clip,width=12.0cm]{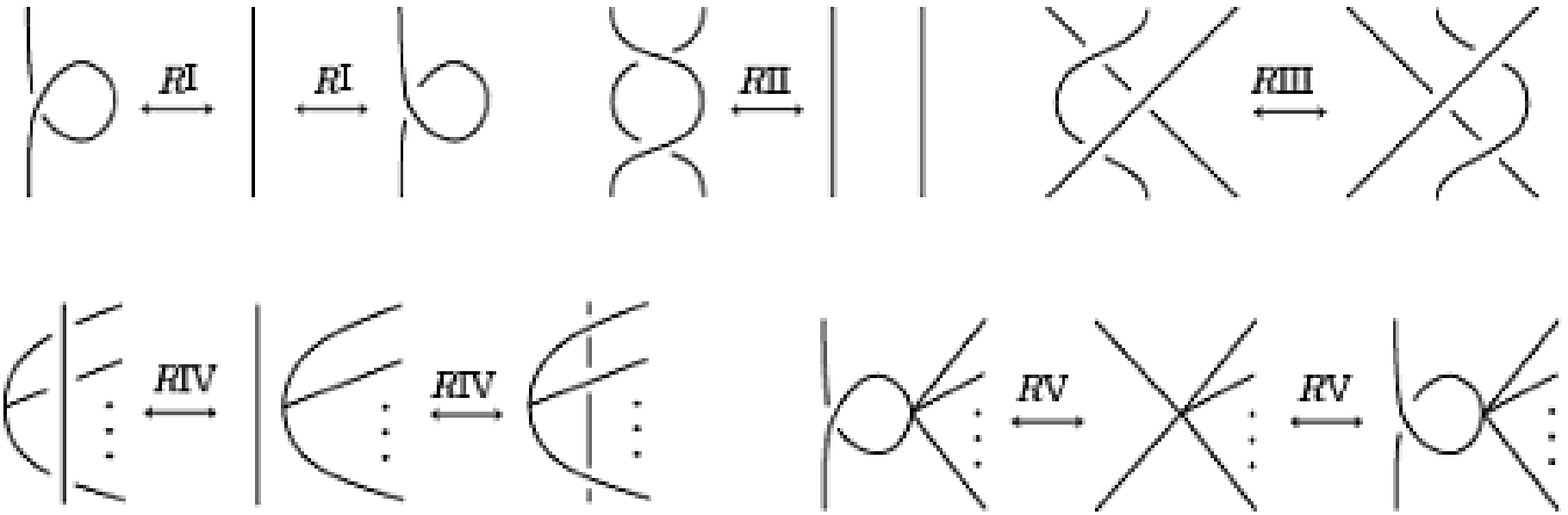}
    \caption{}
    \label{Rmoves}
  \end{center}
\end{figure}

\section{Dehn $p$-colorings of diagrams of spatial Euler graphs}\label{sec:Dehncoloring}
\begin{definition}
Let $D$ be a diagram of a spatial Euler graph and $\mathcal{R}(D)$ the set of regions of $D$.
A {\it Dehn $p$-coloring} of $D$ is a map $C: \mathcal{R}(D) \to \mathbb{Z}_p$ satisfying the following condition: 
\begin{itemize}
\item For a crossing $c$ with regions $r_1, r_2, r_3$ and $r_4$ such that $r_2$ is adjacent to an arbitrary chosen $r_1$ by an under-arc and $r_3$ is adjacent to $r_1$ by the over-arc as depicted in Figure~\ref{VWSIcoloring7},
\begin{align*}\label{crossingcondition}
C(r_1)- C(r_2)+ C(r_3) -C(r_4) = 0
\end{align*}
holds, which we call the {\it crossing condition}.
\end{itemize}
\begin{figure}[h]
  \begin{center}
    \includegraphics[clip,width=7.0cm]{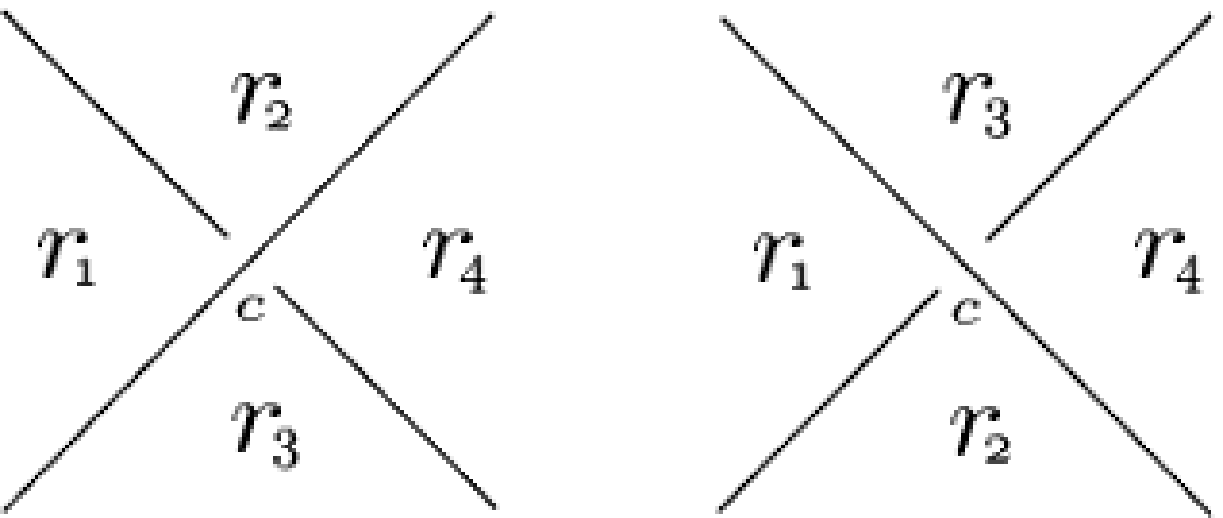}
    \caption{}
    \label{VWSIcoloring7}
  \end{center}
\end{figure}
We call $C(r)$ the {\it color} of a region $r$. 
We denote by ${\rm Col}_{p}(D)$ the set of Dehn $p$-colorings of $D$.
We denote by $(D,C)$ a diagram $D$ equipped with a Dehn $p$-coloring $C$, and we often represent $(D,C)$ by assigning the color $C(r)$ to each region $r$ of $D$.

\end{definition}
\begin{proposition}\label{prop:coloring}
Let $D$ and $D'$ be diagrams of spatial Euler graphs. If $D$ and $D'$ represent the same spatial graph, then there exists a bijection between ${\rm Col}_{p}(D)$ and ${\rm Col}_{p}(D')$.  
\end{proposition}
\begin{proof}
Let $D$ and $D'$ be diagrams such that $D'$ is obtained from $D$ by a single
Reidemeister move. Let $E$ be a $2$-disk in which the move is applied. Let $C$ be a Dehn
$p$-coloring of $D$. We define a Dehn $p$-coloring $C'$ of $D'$, corresponding to
$C$, by $C'(r) = C(r)$ for each region $r$ appearing in the outside of $E$. Then the colors
of the regions appearing in $E$, by $C'$, are uniquely determined, see Figures~\ref{VWSIcoloring8} and \ref{VWSIcoloring9} for Reidemeister moves of type IV and V. 
\end{proof}
\begin{figure}[h]
  \begin{center}
    \includegraphics[clip,width=9.0cm]{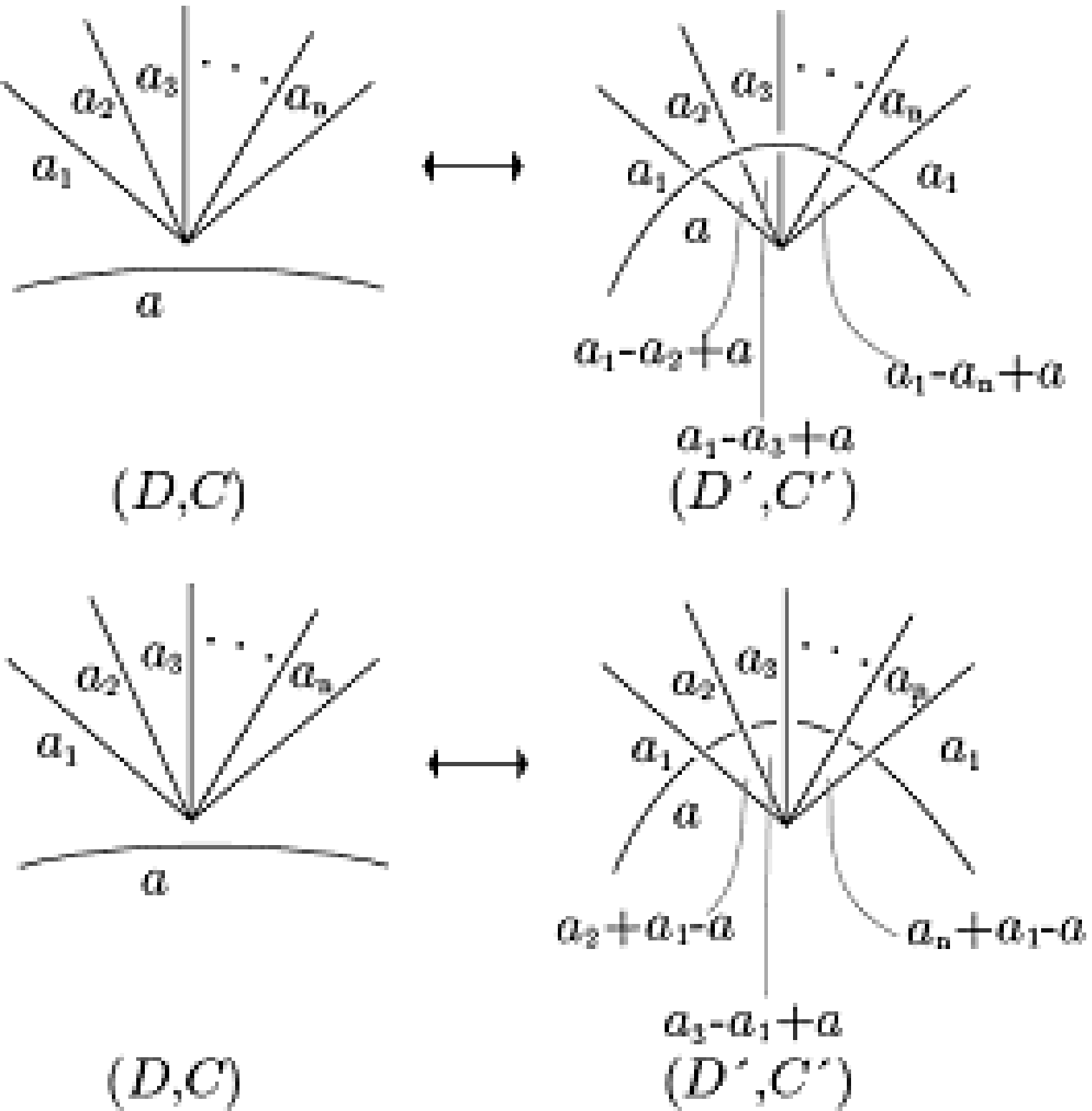}
    \caption{}
    \label{VWSIcoloring8}
  \end{center}
\end{figure}
\begin{figure}[h]
  \begin{center}
    \includegraphics[clip,width=9.0cm]{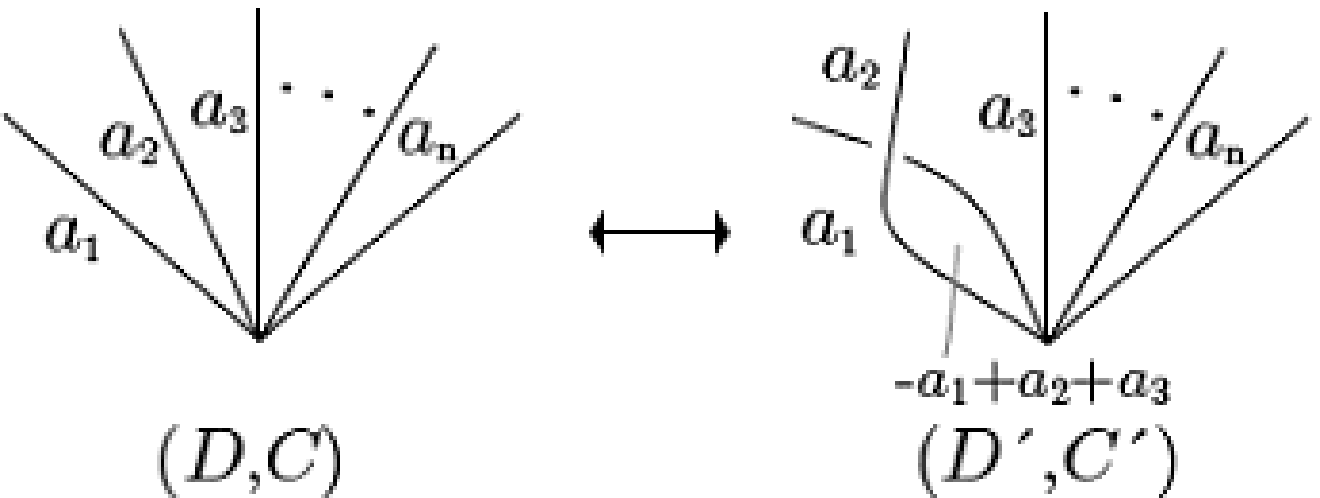}
    \caption{}
    \label{VWSIcoloring9}
  \end{center}
\end{figure}

\noindent Proposition~\ref{prop:coloring} implies that the number of Dehn $p$-colorings, i.e. $\# {\rm Col}_{p}(D)$, is an invariant of spatial Euler graphs.

\begin{remark}\label{rem:Dehncoloring}
For spatial graphs including an odd-valent vertex, even if we add any vertex condition, there does not exist a Dehn $p$-coloring such that $\# {\rm Col}_{p}(D)$ is an invariant of special graphs. This is because $\# {\rm Col}_{p}(D)$ might be changed under the Reidemeister move of type IV depicted in the lower picture of Figure~\ref{VWSIcoloring8}.
\end{remark}

\section{Vertex-weight invariants of spatial Euler graphs} \label{sec:VWSI}
Let $f : U_p \to S$ be an invariant under the transformations  (Op1)-(Op4) of Definition~\ref{def:Requiv}. 
For example, we can define $f$ under some proper situation by $f=\tau_p$, $f=\varepsilon_p$, $f=\mu_p$, $f=\mu_{p,\tau}$, $f=\tau_p \times \varepsilon_p, \ldots ,$ or $f=\tau_p \times  \varepsilon_p \times  \mu_p \times  \mu_{p,\tau}$.

Let $D$ be a diagram of a spatial Euler graph $G$ and $C\in {\rm Col}_{p}(D)$. 
For a vertex $v$ of $D$ with regions $r_1, r_2, \ldots, r_n $  in clockwise direction as shown in Figure~\ref{VWSIcoloring6}, we take a weight $W_{f}(D, C; v)$ as $W_{f}(D, C; v)=\Big(n, f\big(C(r_1), C(r_2),  \ldots, C(r_n)\big)\Big),$
where in this paper, we represent it by 
$$W_{f}(D, C; v)=\Big({\rm valency}=n, f=f\big(C(r_1), C(r_2),  \ldots, C(r_n)\big)\Big)$$
as an easy-to-understand way.
\begin{figure}[t]
  \begin{center}
    \includegraphics[clip,width=4.0cm]{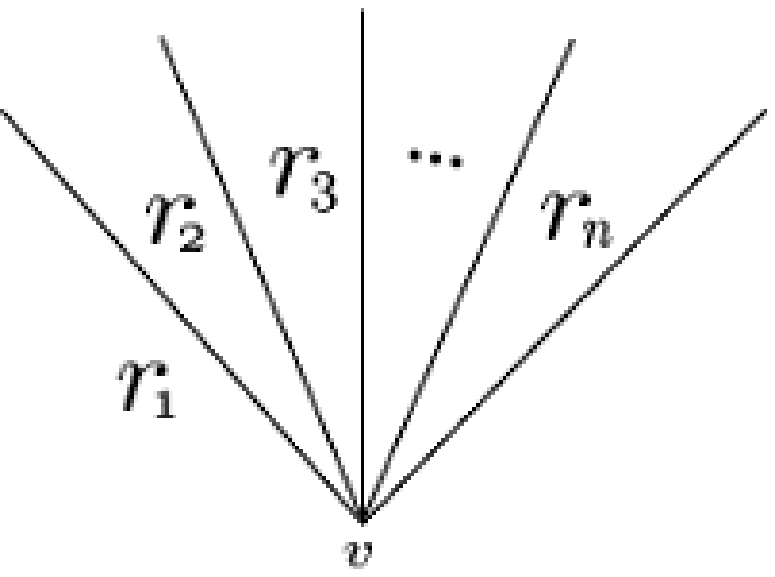}
    \caption{}
    \label{VWSIcoloring6}
  \end{center}
\end{figure}
We denote by $W_{f}(D,C)$ the multi-set of the weights of all vertices of $D$.
As a multi-set, set 
\[
\Phi_{f}(D)= \Big\{W_f(D,C) ~\Big|~ C\in {\rm Col}_{p}(D) \Big\},
\]
which we call the {\it vertex-weight invariant} of $D$ (or $G$) with respect to $f$. Then we have the following theorem:
\begin{theorem}\label{thm:VWSI}
Let $D$ and $D'$ be diagrams of spatial Euler graphs.
If $D$ and $D'$ represent the same spatial graph, then we have $\Phi_{f}(D)= \Phi_{f}(D')$.
That is, $\Phi_{f}(D)$ is an invariant for spatial Euler graphs.
\end{theorem}
\begin{proof}
First, we note that $W_f(D, C; v)$ does not depend on the starting region $r_1$ to read the regions $r_1, \ldots , r_n$ in clockwise direction around $v$ since $f$ is unchanged under the transformation (Op1).

Now we show that $W_f(D, C; v)$ is unchanged under the Reidemeister moves of spatial graph diagrams. 
Let $D$ and $D'$ be diagrams such that $D'$ is obtained from $D$ by a single
Reidemeister move shown in Figure~\ref{VWSIcoloring9}. Let $C$ be a Dehn $p$-coloring of $D$, and $C'$ the corresponding Dehn $p$-coloring of $D'$.
For an $n$-valent vertex $v$ of $(D,C)$ with $W_f(D, C; v)=\big({\rm valency}=n, f=f(a_1, a_2, \ldots , a_n)\big)$ as shown in the left of Figure~\ref{VWSIcoloring9}, the corresponding vertex, say $v'$, of $(D',C')$ has $W_f(D', C'; v')=\big({\rm valency}=n, f=f(a_1, -a_1+a_2+a_3, a_3,  \ldots , a_n)\big)$. 
Since $f$ is unchanged under the transformation (Op4) of Definition~\ref{def:Requiv}, we have $W_{f}(D', C'; v')=W_{f}(D, C; v)$. The same argument applies to the cases of the other Reidemeister moves.
This implies that $W_{f}(D,C)= W_{f}(D',C')$, which leads to the property that $\Phi_{f}(D)=\Phi_{f}(D')$.
\end{proof}

\begin{remark}
We may arrange the definition of $W_{f}(D,C)$ for examples as follows: Define $W_{f}(D, C; v)$ by $W_{f}(D, C; v)=f\big(C(r_1), C(r_2),  \ldots, C(r_n)\big)$, and define $W_{f}(D,C)$ by the multiset of the weights of all vertices of $D$.  In another way, 
define $W_{f}(D, C; v)$ by $W_{f}(D, C; v)=f\big(C(r_1), C(r_2),  \ldots, C(r_n)\big)$, and define $W_{f}(D,C)$ by the sum of the weights of all vertices of $D$ when the target set $S$ of $f$ has an Abelian group structure. 
\end{remark}

\begin{example}
Let $G$ and $G'$ be the spatial graphs depicted in Figure~\ref{fig:GG'}.
The spatial graphs $G$ and $G'$ can be distinguished with neither their constituent links nor $\# {\rm Col}_{p}$, and indeed, $\# {\rm Col}_{p}(G)=\# {\rm Col}_{p}(G')=p^5$ holds for any $p\in \mathbb Z_{\geq 2}$. 
On the other hand, we can distinguish them with $\Phi_{\tau_{3}}$, which is 
shown as follows:
We show that the multiset $\big\{({\rm valency}=4,\tau_{3}=1),({\rm valency}=6,\tau_{3}=3)\big\} $ is not included in $\Phi_{\tau_3}(G)$, while it is included in $\Phi_{\tau_3}(G')$.
For the diagram $D$,  depicted in the left of Figure~\ref{VWSIgraphs1},  of $G$ and a Dehn $3$-coloring $C\in {\rm Col}_{3}(D)$, assume that $W_{\tau_3}(D,C; v_2)=3$ for the 6-valent vertex $v_2$. Then the regions $r_1$-$r_6$ around $v_2$ must be colored alternately as $C(r_1)=a, C(r_2)=b, C(r_3)=a, C(r_4)=b, C(r_5)=a, C(r_6)=b$ for some $a,b\in \mathbb Z_3$. Put $C(r_7)=c$ for some $c\in \mathbb Z_3$. We then have $C(r_8)=-a+b+c$, $C(r_9)=2a-c$, $C(r_{10})=a+b-c$, $C(r_{11})=2a-c$ and $C(r_{12})=c$ from the crossing conditions at $c_1$, $c_2$, $c_3$, $c_4$ and $c_5$, respectively, in this order, see the right of Figure~\ref{VWSIgraphs1}. 
Then the crossing condition at $c_6$ shows that 
$$c=C(r_{12})= C(r_9)-C(r_8)+C(r_{10})= (2a-c)-(-a+b+c)+(a+b-c) = a,$$
which implies that the regions $r_1$, $r_6$, $r_{12}$, $r_2$ around $v_1$ must be colored alternately, that is, we have $$W_{\tau_3}(D,C; v_1)=\tau_3\big(C(r_1), C(r_6), C(r_{12}), C(r_2)\big)=\tau_3(a,b,a,b)=3.$$  
We note that the crossing condition at $c_7$ is satisfied when $a=c$.
Therefore the multiset $\big\{({\rm valency}=4,\tau_{3}=1),({\rm valency}=6,\tau_{3}=3)\big\} $ is not included in $\Phi_{\tau_3}(G)$.
On the other hand, the multiset $\big\{({\rm valency}=4,\tau_{3}=1),({\rm valency}=6,\tau_{3}=3)\big\} $ is included in $\Phi_{\tau_3}(G')$ since the Dehn $3$-colored diagram of $G'$ in Figure~\ref{VWSIgraphs2} gives this multiset.

Indeed, we have  
\[
\Phi_{\tau_3}(G)=\left.
\begin{cases}
\Big\{({\rm valency}=4,\tau_{3}=1),({\rm valency}=6,\tau_{3}=1)\Big\}(216 {\rm times}), \\
\Big\{({\rm valency}=4,\tau_{3}=3),({\rm valency}=6,\tau_{3}=1)\Big\}(18 {\rm times}), \\
\Big\{({\rm valency}=4,\tau_{3}=3),({\rm valency}=6,\tau_{3}=3)\Big\}(9 {\rm times})
\end{cases}
\right\}, 
\]
and
\[
\Phi_{\tau_3}(G')=\left.
\begin{cases}
\Big\{({\rm valency}=4,\tau_{3}=1),({\rm valency}=6,\tau_{3}=1)\Big\}(144 {\rm times}), \\
\Big\{({\rm valency}=4,\tau_{3}=1),({\rm valency}=6,\tau_{3}=3)\Big\}(18 {\rm times}), \\
\Big\{({\rm valency}=4,\tau_{3}=3),({\rm valency}=6,\tau_{3}=1)\Big\}(72 {\rm times}), \\
\Big\{({\rm valency}=4,\tau_{3}=3),({\rm valency}=6,\tau_{3}=3)\Big\}(9 {\rm times})
\end{cases}
\right\}.
\]
Thus $G$ and $G'$ can be distinguished by $\Phi_{\tau_{3}}$.
\end{example}
\begin{figure}[ht]
  \begin{center}
    \includegraphics[width=9cm]{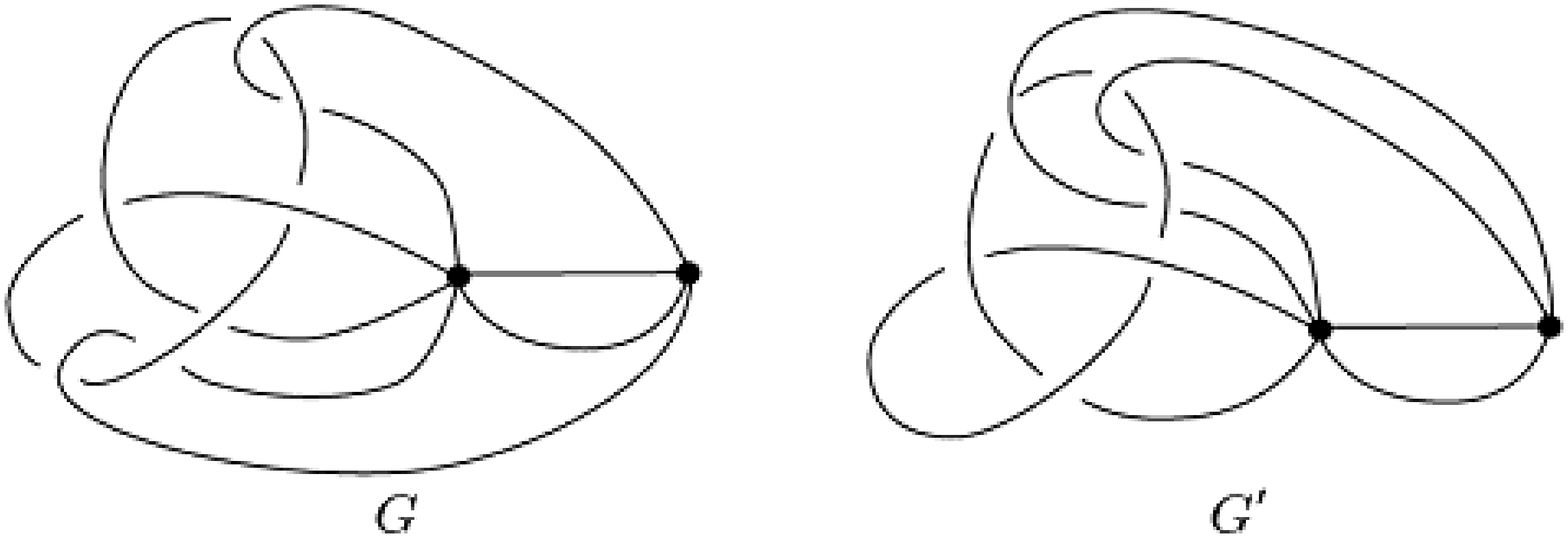}
    \caption{}
    \label{fig:GG'}
  \end{center}
\end{figure}
\begin{figure}[ht]
  \begin{center}
    \includegraphics[width=9cm]{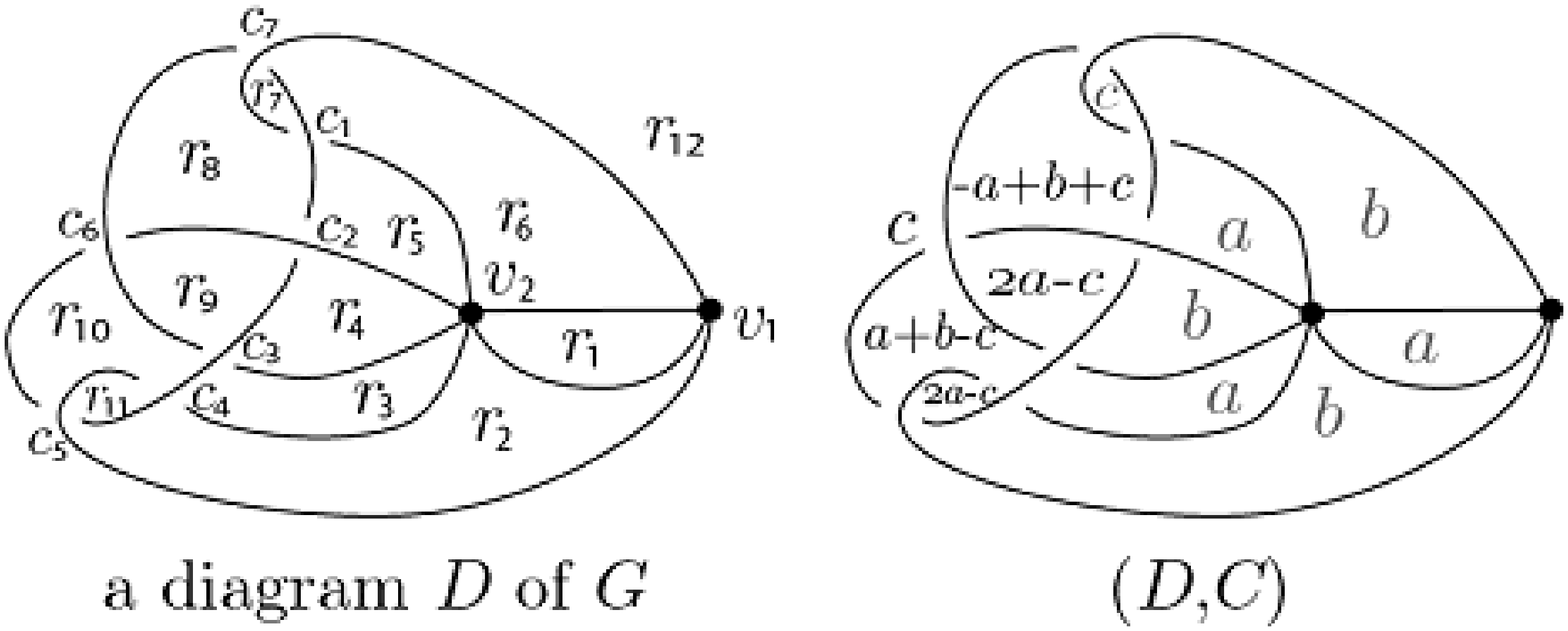}
    \caption{}
    \label{VWSIgraphs1}
  \end{center}
\end{figure}
\begin{figure}[ht]
  \begin{center}
    \includegraphics[width=4.3cm]{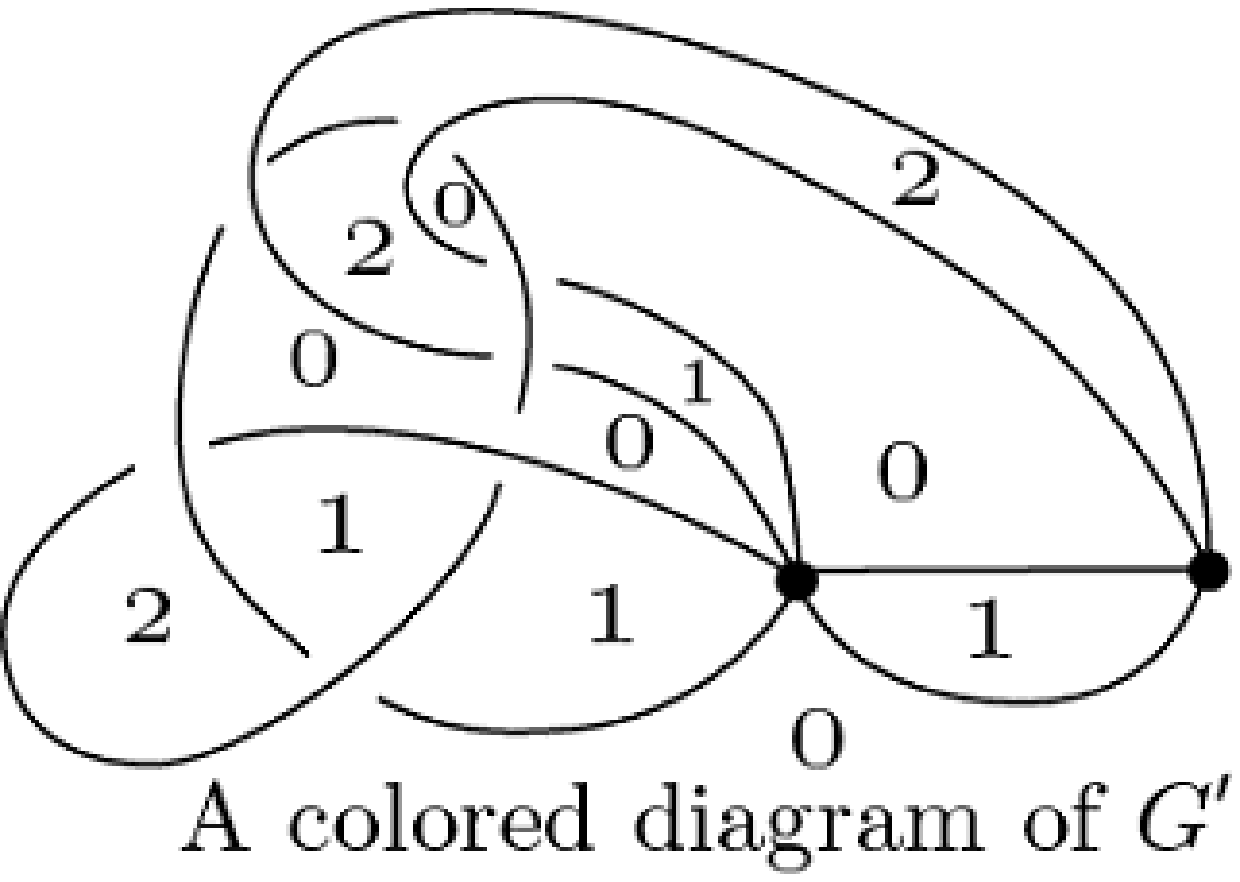}
    \caption{}
    \label{VWSIgraphs2}
  \end{center}
\end{figure}

\begin{remark}\label{rem:VWSI}
As mentioned in Remark \ref{rem:Dehncoloring}, for spatial graphs including an odd-valent vertex, even if we add any vertex condition, there does not exist a vertex-weight invariant such that $\Phi_{f}(D)$ is an invariant of special graphs. This is because $\Phi_{f}(D)$ might be changed  under the Reidemeister move of type IV depicted in the lower picture of Figure~\ref{VWSIcoloring8}.
\end{remark}

\section{An application for spatial graphs with odd-valent vertices}
In Section~\ref{sec:Dehncoloring}, we introduced the coloring invariants $\#{\rm Col}_{p}(G)$ of spatial Euler graphs, and in Section~\ref{sec:VWSI}, we introduced the invariants $\Phi_{f}(G)$ of spatial Euler graphs. 
As mentioned in Remarks~\ref{rem:Dehncoloring} and \ref{rem:VWSI}, 
for spatial graphs including odd-valent vertices, these invariants are meaningless.
However, for spatial graphs with odd-valent vertices, we  can apply our invariants to the spatial Euler graphs obtained by  taking 
 the parallel of some edges, and thus, our invariants can be also useful for spatial graphs with odd-valent vertices, where this method was introduced in \cite{IshiiYasuhara97}. 
In this section, we show how to apply our invariants to spatial graphs with odd-valent vertices in detail,  and give some calculation example.

Let $G$ be a spatial graph and $E_G=\{e_1, \ldots e_s\}$ the set of  edges of $G$.  
We replace $r$ edges $e_{i_1}, \ldots, e_{i_r} \in E_G$  without duplicates with the doubles of the edges, respectively, where the double of an edge $e$ is obtained by replacing the edge $e$ with the $2$-parallel $e_+ \cup e_-$ of $e$ such that  $e_+$ and $e_-$ are connected at the end points of $e$ as in Figure~\ref{VWSIdoubleedge}. 
Here one might think that there is an ambiguity for twists of the parallel edges $e_+$ and $e_-$, and however, the ambiguity can be solved by  equivalence transformations of spatial graphs corresponding to the Reidemeister moves of type V.   
\begin{figure}
  \begin{center}
    \includegraphics[width=9cm]{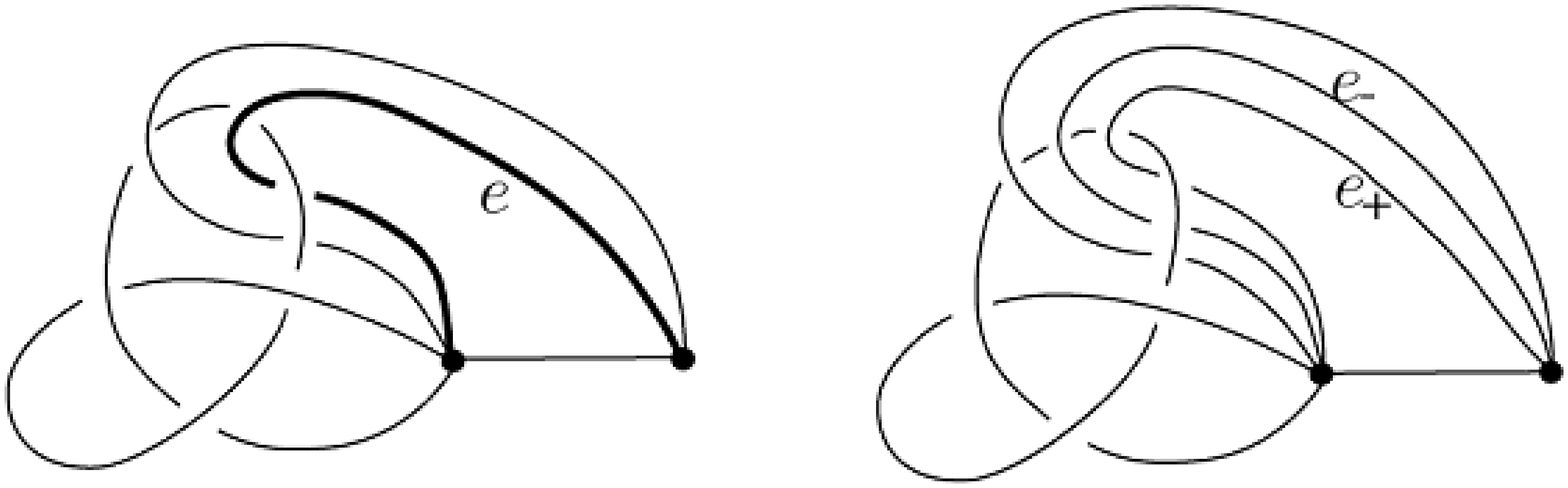}
    \caption{}
    \label{VWSIdoubleedge}
  \end{center}
\end{figure}
We call the resultant spatial graph the {\it double of $G$ with respect to $\{e_{i_1}, \ldots, e_{i_r}\}$}, and we denote it by $d(G; \{e_{i_1}, \ldots, e_{i_r}\})$. 
We say that $\{e_{i_1}, \ldots, e_{i_r}\}$ is {\it doublable} if $d(G; \{e_{i_1}, \ldots, e_{i_r}\})$ is a spatial Euler graph. 
For $r \in \mathbb Z_{+}$, we set
\[
d\hspace{-0.3mm}{\rm Edges}(G;r)=\left\{
\{e_{i_1}, \ldots, e_{i_r}\} \subset E_G ~\Big|~ 
\begin{array}{l}
\{e_{i_1}, \ldots, e_{i_r}\} \mbox{ is doublable, and}\\
\#\{e_{i_1}, \ldots, e_{i_r}\}=r
\end{array}
\right\}.
\]
Let $\Phi$ be an invariant of spatial Euler graphs.
For a spatial graph $G$ that might have an odd-valent vertex and $r\in \mathbb Z_+$, set 
 \[
\Phi(G;r)= \big\{\Phi\big(d(G; \{e_{i_1}, \ldots, e_{i_r}\})\big)  ~\big|~\{e_{i_1}, \ldots, e_{i_r}\} \in d\hspace{-0.3mm}{\rm Edges}(G;r)\big\}
 \]
as a multiset.
We have the following proposition:
\begin{proposition}
$\Phi(G;r)$ is an invariant of spatial graphs.
\end{proposition}
\begin{proof}
Put
\[
d(G;r)=\big\{d(G; \{e_{i_1}, \ldots, e_{i_r}\})  ~\big|~\{e_{i_1}, \ldots, e_{i_r}\} \in d\hspace{-0.3mm}{\rm Edges}(G;r)\big\}.
\]
We then have that if two spatial graphs $G_1$ and $G_2$  are equivalent, then there exists a bijection $\psi: d(G_1;r) \to d(G_2;r)$ such that for $G_1'\in d(G_1;r)$ and $G_2' =\psi(G_1')$, $G_1'$ and $G_2'$ are equivalent, see \cite{IshiiYasuhara97}. 
Besides, by Theorem~\ref{thm:VWSI}, we have $\Phi(G_1')= \Phi(G_2')$.
This completes the proof.
\end{proof}
\begin{example}
Let $G_1$ and $G_2$ be the spatial graphs depicted in Figure~\ref{G1G2}.
We note that they have odd-valent vertices. Let us consider the case that $\Phi=\Phi_{\tau_3}$ and $r=1$ for $\Phi(G_1;r)$ and $\Phi(G_2;r)$.
Since $d{\rm Edges}(G_1;1)=\big\{\{e_{11}\}, \{e_{12}\}, \{e_{13}\} \big\}$ and $d{\rm Edges}(G_2;1)=\big\{\{e_{21}\}, \{e_{22}\}, \{e_{23}\} \big\}$ as depicted in Figure~\ref{G1G2}, 
\begin{align*}
&\Phi_{\tau_3}(G_1;1) = \big\{\Phi_{\tau_3}\big(d(G_1; \{e_{11}\})\big), \Phi_{\tau_3}\big(d(G_1; \{e_{12}\})\big), \Phi_{\tau_3}\big(d(G_1; \{e_{13}\})\big)  \big\}, \mbox{ and }\\
&\Phi_{\tau_3}(G_2;1) = \big\{\Phi_{\tau_3}\big(d(G_2; \{e_{21}\})\big), \Phi_{\tau_3}\big(d(G_2; \{e_{22}\})\big), \Phi_{\tau_3}\big(d(G_2; \{e_{23}\})\big)  \big\},
\end{align*}
see Figures~\ref{G1bcd} and \ref{G2bcd}.
Since we have 
\[
\big\{\Phi_{\tau_3}\big(d(G_1; \{e_{11}\})\big), \Phi_{\tau_3}\big(d(G_1; \{e_{12}\})\big), \Phi_{\tau_3}\big(d(G_1; \{e_{13}\})\big)  \big\}=\big\{\mathbb{A}, \mathbb{A}, \mathbb{A}\big\}
\]
and
\[
\big\{\Phi_{\tau_3}\big(d(G_2; \{e_{21}\})\big), \Phi_{\tau_3}\big(d(G_2; \{e_{22}\})\big), \Phi_{\tau_3}\big(d(G_2; \{e_{23}\})\big)  \big\}=\big\{\mathbb{A}, \mathbb{A}, \mathbb{B}\big\}
, 
\]
it holds that $\Phi_{\tau_3}(G_1;1)\not = \Phi_{\tau_3}(G_2;1)$, 
where
\[\mathbb{A}=\left.
\begin{cases}
\Big\{({\rm valency}=4,\tau_3=1),({\rm valency}=6,\tau_3=1)\Big\}(216 {\rm times}), \\
\Big\{({\rm valency}=4,\tau_3=3),({\rm valency}=6,\tau_3=1)\Big\}(18 {\rm times}), \\
\Big\{({\rm valency}=4,\tau_3=3),({\rm valency}=6,\tau_3=3)\Big\}(9 {\rm times})
\end{cases}
\right\}\]
and
\[\mathbb{B}=\left.
\begin{cases}
\Big\{({\rm valency}=4,\tau_3=1),({\rm valency}=6,\tau_3=1)\Big\}(144 {\rm times}), \\
\Big\{({\rm valency}=4,\tau_3=1),({\rm valency}=6,\tau_3=3)\Big\}(18 {\rm times}), \\
\Big\{({\rm valency}=4,\tau_3=3),({\rm valency}=6,\tau_3=1)\Big\}(72 {\rm times}), \\
\Big\{({\rm valency}=4,\tau_3=3),({\rm valency}=6,\tau_3=3)\Big\}(9 {\rm times})
\end{cases}
\right\}. \]
Thus $G_1$ and $G_2$ can be distinguished by $\Phi_{\tau_3}(\,\bullet\,; 1)$.
We note that $G_1$ and $G_2$ can be distinguished by neither their constituent links nor the coloring numbers $\#{\rm Col}_p(\,\bullet\,;1)$ for any $p\in \mathbb Z_{\geq 2}$, where $\#{\rm Col}_p(G_1; 1) = \#{\rm Col}_p(G_2; 1)=\{p^5, p^5, p^5 \}$.

\end{example}
\begin{figure}
  \begin{center}
    \includegraphics[width=9cm]{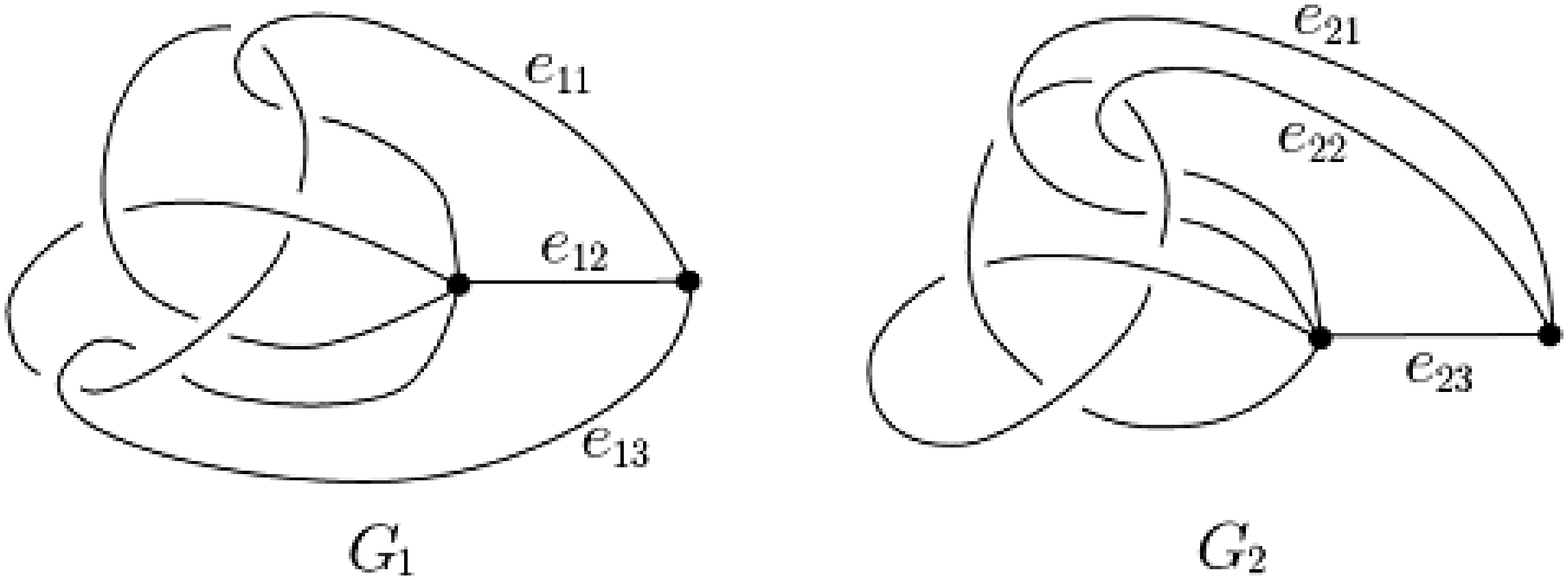}
    \caption{}
    \label{G1G2}
  \end{center}
\end{figure}
\begin{figure}
  \begin{center}
    \includegraphics[width=11cm]{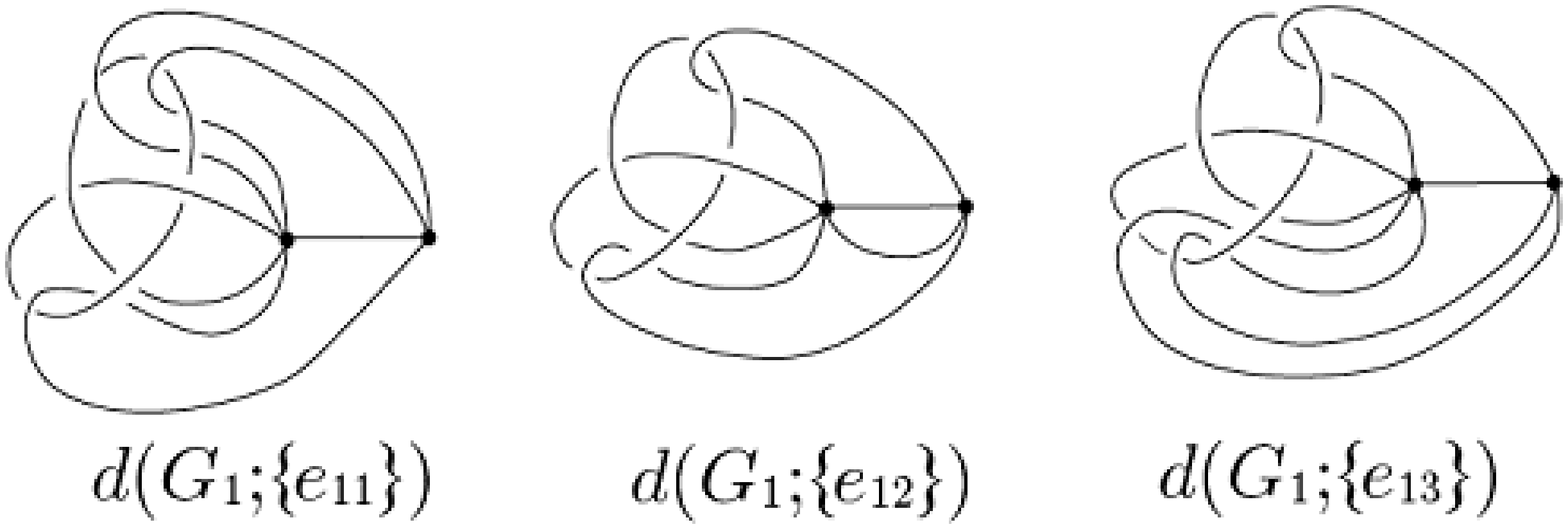}
    \caption{}
    \label{G1bcd}
  \end{center}
\end{figure}
\begin{figure}
  \begin{center}
    \includegraphics[width=11cm]{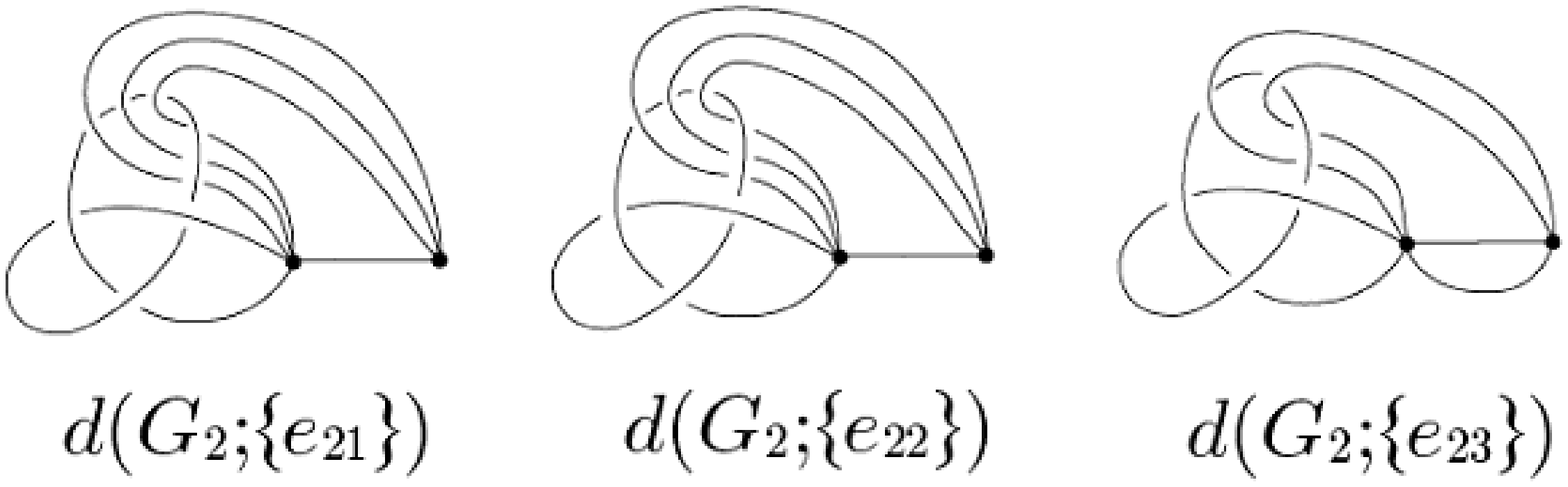}
    \caption{}
    \label{G2bcd}
  \end{center}
\end{figure}

\section*{Acknowledgments}

The first author was supported by JSPS KAKENHI Grant Number 16K17600.

\end{document}